\documentclass[11pt]{amsart}
\usepackage{amssymb,amscd,amsmath,enumerate}
\usepackage[all]{xy}
\newtheorem{thm}{Theorem}[section]

\newtheorem{prop}[thm]{Proposition}
\newtheorem{lem}[thm]{Lemma}
\theoremstyle{remark}

\newtheorem{question}[thm]{Question}

\theoremstyle{definition}

\numberwithin{equation}{section}

\renewcommand{\bar}{\overline}

\newcommand{\A}{{\mathbb{A}}}

\newcommand{\F}{{\mathbb{F}}}

\newcommand{\Z}{{\mathbb{Z}}}

\newcommand{\uW}{{\underline{W}}}

\newcommand{\uC}{{\underline{C}}}
\newcommand{\uG}{{\underline{G}}}
\newcommand{\uU}{{\underline{U}}}

\newcommand{\uX}{{\underline{X}}}
\newcommand{\uY}{{\underline{Y}}}
\newcommand{\uZ}{{\underline{Z}}}
\newcommand{\cG}{{\mathcal{G}}}

\newcommand{\sG}{{\sf {G}}}
\newcommand{\sS}{{\sf {S}}}
\newcommand{\sX}{{\sf {X}}}

\renewcommand{\Pr}{\mathbf {P}}

\newcommand{\Gal}{\mathrm{Gal}}

\newcommand{\SU}{\mathrm{SU}}

\newcommand{\Span}{\mathrm{Span}}
\newcommand{\Stab}{\mathrm{Stab}}

\newcommand{\Spec}{\mathrm{Spec}\;}

\newcommand{\im}{\mathrm{im}}

\newcommand{\bg}{\mathbf g}

\newcommand{\Alt}{{\raise 2pt\hbox{$\scriptstyle\bigwedge$}}}

\begin{document}
\title{Most Words are Geometrically Almost Uniform}

\author{Michael Larsen}
\email{mjlarsen@indiana.edu}
\address{Department of Mathematics\\
    Indiana University \\
    Bloomington, IN 47405\\
    U.S.A.}

\begin{abstract}
If $w$ is a word in $d>1$ letters and $G$ is a finite group, evaluation of $w$ on a uniformly randomly chosen 
$d$-tuple in $G$ gives a random variable with values in $G$, which may or may not be uniform.
It is known \cite{LST} that if $G$ ranges over finite simple groups of given root system and characteristic, a positive proportion of words $w$
give a distribution which approaches uniformity in the limit as $|G|\to \infty$.  In this paper, we show that the proportion is in fact $1$.

\end{abstract}


\thanks{The author was partially supported by NSF grant DMS-1702152.}

\maketitle

\section{Introduction}

A \emph{word} for the purposes of this paper is an element of the
free group $F_d$. For any finite group $G$, the word $w$ defines a
word map $w_G\colon G^d \to G$ by substitution; we denote it $w$ when $G$ is understood. 
If $U_G$ defines the uniform measure on $G$, we can measure the failure of random values of $w$ to
be uniform by comparing the pushforward $w_* U_{G^d}$ to the uniform distribution $U_G$.
We say $w$ is \emph{almost uniform} for an infinite family of finite groups $G$ if 
$$\lim_{|G|\to \infty} \Vert w_* U_{G^d}-U_G\Vert = 0,$$
where $\Vert\cdot\Vert$ denotes the $L^1$ norm, and $G$ ranges over the groups of the family.
We are particularly interested in the family of finite simple groups.  

When $w$ is of the form $w_0^n$ for some $n\ge 2$, then $w$ is said to be a \emph{power word}.
It is easy to see that power words are not almost uniform for finite simple groups; for instance, in large symmetric groups,
most elements are not $n$th powers at all  \cite{Pouyanne}.  There has been speculation as to whether all non-power words 
are almost uniform for finite simple groups (see, e.g., \cite[Problem~4.7]{Shalev} and \cite[Question~3.1]{Larsen14}).  Since power words are exponentially thin \cite{LM},  one could ask an
easier question: is the set of words which are not almost uniform for finite simple groups thin?  Or, easier still, does it have density $0$?
Some words are known to be almost uniform for finite simple groups:
primitive words, which are exactly uniform for all groups; the commutator word $x_1 x_2 x_1^{-1} x_2^{-1}$ by \cite{GS}, words of the
form $x_1^m x_2^n$ by \cite{LS4}, and, recently, all words of \emph{Waring type}, i.e., words which can be written as a product of two non-trivial words involving disjoint variables \cite[Theorem~1]{LST}.
The fundamental group of any orientable genus $g$ surface is therefore covered for all $g\ge 1$, and, more generally, various words in which some variables appear exactly twice can also be treated by combining the idea of 
Parzanchevski-Schul \cite{PS} with the method of Liebeck-Shalev \cite{LiS}
All of these words, of course, are in some sense rare and atypical.

From the point of view of algebraic geometry, the easiest families of finite simple groups to consider are those of the form $\uG(\F_{q^n})/Z(\uG(\F_{q^n}))$, where $\uG$ is a simple, simply connected algebraic group over $\F_q$, and $n$ ranges over the positive integers.  We say that $w$ is \emph{geometrically almost uniform} for $\uG$ if it is so for this family of groups. 
In \cite[Theorem~2]{LST}, it is proved that this property is equivalent to an algebro-geometric condition on $w$, namely that
the morphism of varieties $w_{\uG}\colon \uG^d\to \uG$ (which by a theorem of Borel \cite{Borel} is dominant) has geometrically irreducible generic fiber.
Using this criterion, it is proved in \cite[Theorem~3]{LST} that for each $d$, there exists a set of words of density greater than $1/3$ which are almost uniform for $\uG$ for all $\uG/\F_q$.  (Note that this does not imply that these words are almost uniform for the family of all finite simple groups of Lie type.)

The main result of this paper is that for each $\uG$ the set of words which are geometrically almost uniform for $\uG$ has density $1$.  More explicitly:

\begin{thm}
\label{main}
Let $\F_q$ and $\uG$ be fixed.  Let $i_1,i_2,\ldots$ be chosen independently and uniformly from $\{1,\ldots,d\}$.  Let $w = x_{i_1}\cdots x_{i_n}$ be a random word of length $n$ defined in this way.
Then the probability that $w$ is geometrically almost uniform for $\uG$ goes to $1$ as $n\to \infty$.

\end{thm}

The idea of the proof is as follows.  
In \cite[Corollary~2.3]{LST}, it is proved that if the image $\bar w$ of $w$ under the abelianization map $F_d\to \Z^d$ is primitive, i.e., if $\gamma(\bar w) = 1$, where $\gamma$ denotes the g.c.d. of its coordinates,
then $w$ is almost uniform for every $\uG$, the idea being that $w_{\uG(\F_{q^n})}$ is then surjective for all $n$, and this implies that $w_{\uG}$ does not factor through a non-birational generically finite 
morphism $\uX_0\to \uG$.

Now, the image of a random walk on $F_d$ under the abelianization map is a random walk on $\Z^d$.  
If $\sX_{d,n}$ is the endpoint of a random walk of length $n$ on $\Z^d$, then 
$$\limsup_{n\to \infty} \Pr[\gamma(\sX_{d,n})=1] < 1$$ 
for all $d$, so this is not good enough to get a result which covers almost all words.  A new idea is needed.

By a probabilistic analysis, we prove that for each $d$,
$$\lim_{M\to \infty} \liminf_{n\to \infty} \Pr[1\le \gamma(\sX_{d,n})\le M] = 1.$$
Thus, it suffices to prove that for each $d\ge 2$ and $m>0$, in the limit as $n$ goes to infinity, the fraction of $w$ of length $n$ with $\gamma(\bar w) = m$ for which $w$ is almost uniform in rank $\le r$ goes to $1$.
For any such $w$ and any group $G$, the image of $w_G$ contains all $m$th powers in $G$.  For $m>1$, this no longer implies geometric irreducibility of the generic fiber of $w_{\uG}$, but it puts very strong constraints on which quasi-finite morphisms $\uX_0\to \uG$ it can factor through.

To see how to exploit such constraints, consider the following toy problem.  Suppose $f\colon \A^1\to \A^1$ is defined over $\F_q$; for all $n$, $f(\F_{q^n})$ contains all squares in $\F_{q^n}$; and for some $n$, $f(\F_{q^n})$ contains a non-square.  We claim this implies $\deg f=1$.  

Indeed, consider the curve $\uC\colon y^2 = f(x)$.  The conditions on the image of $f$ imply that $\uC$ is geometrically irreducible and therefore has $(1+o(1))q^n$ points over $\F_{q^n}$.  Consider the morphism of degree $\deg f$ from $\uC$ to the affine line given by the function $y$.  By the Chebotarev density theorem for finite extensions of $\F_q(t)$, in the limit as $n\to \infty$, a fixed positive proportion of points in $\A^1(\F_{q^n})$ have preimage in $\uC(\F_{q^n})$ consisting of $\deg f$ points.  Since the $y$-map is surjective on $\F_{q^n}$-points, this implies that $\deg f=1$.

To apply this idea in the word map setting, one needs to find elements in $w(\uG(\F_{q^n})^d)$ which play the role of non-square elements in $f(\F_{q^n})$.  We do not need to find them for all $w$, just for almost all in an asymptotic sense.  An approach to achieving this is to fix a $d$-tuple $\bg\in \uG(\F_{q^n})^d$ and estimate the probability that $w(\bg)$ is a ``non-square'' element.
For large enough $n$, one can view $w(\bg)$ as uniformly distributed in $\uG(\F_{q^n})$.  In order to get the probability of success to approach $1$, it is necessary to use not a single $\bg$ but a sufficiently large number of independent choices $\bg_1,\ldots,\bg_N$.  The existence of $N$ elements of $\uG(\F_{q^n})^d$ which are independent in this sense (in the limit $n\to\infty$) depends on $\uG(\F_{q^n})^N$ being $d$-generated.  There is a substantial literature,
going back to work of Philip Hall \cite{Hall}, concerning the size of minimal generating sets of $G^N$, where $G$ is a finite simple group.  We use a recent result of Mar\'oti and Tamburini \cite{MT}.

\section{Varieties over Finite Fields}

Throughout this section, a \emph{variety} will always mean a geometrically integral affine scheme of finite type over a finite field.  
Let $A\subset B$ be an inclusion of finitely generated $\F_q$-algebras such $\uX := \Spec A$ and $\uY := \Spec B$ are normal varieties.
Let $\phi\colon \uY\to \uX = \Spec A$ correspond to the inclusion $A\subset B$.  Let $K$ and $L$ denote the fraction fields of $A$ and $B$ respectively.
Let $K_0$ denote the separable closure of $K$ in $L$, which is a finite extension of $K$ since $L$ is finitely generated.
Let $A_0$ denote the integral closure of $A$ in $K_0$, $\uX_0$ the spectrum of $A_0$, and $\psi\colon \uX_0\to \uX$ the morphism corresponding to the inclusion $A\subset A_0$.
As $B\supset A$ is integrally closed in $L\supset K_0$ it follows that $B$ contains $A_0$, so $\phi$ factors through $\psi$.
\begin{prop}
\label{image-size}
For all positive integers $n$, 
\begin{equation}
\label{inclusion}
\phi(\uY(\F_{q^n})) \subset \psi(\uX_0(\F_{q^n})),
\end{equation}
and
\begin{equation}
\label{close}
| \psi(\uX_0(\F_{q^n}))| - |\phi(\uY(\F_{q^n}))| = o(q^{n\dim\uX}).
\end{equation}
Moreover $\psi$ is an isomorphism if and only if $\phi$ has geometrically irreducible generic fiber; if not, there exists $\epsilon > 0$ and a positive integer $m$ such that
\begin{equation}
\label{not-onto}
\psi(\uX_0(\F_{q^n}))| < (1-\epsilon)q^{n\dim\uX}
\end{equation}
if $m$ divides $n$.
\end{prop}

\begin{proof}
As $A\subset A_0\subset B$, the morphism $\phi$ factors through $\psi$,
implying (\ref{inclusion}).

By \cite[Proposition~4.5.9]{EGA42}, $K=K_0$ if and only if the generic fiber of $\phi$ is geometrically irreducible.
By the same proposition, the generic fiber of $\uY\to\uX_0$ is always geometrically irreducible.  By \cite[Th\'eor\`eme~9.7.7]{EGA43}, there is a dense open subset of $\uX_0$
over which the fibers of $\uY\to\uX_0$ are all geometrically irreducible.  Let $\uC$ denote the complement of this subset, endowed with its structure of reduced closed subscheme of $\uX_0$.

It is well known that the Lang-Weil estimate is uniform in families.  There does not seem to be a canonical reference for this fact, but a proof is sketched, for instance in \cite[Proposition~3.4]{LS3}  and in \cite[Theorem~5]{Tao}.
From this, it follows that if $n$ is sufficiently large, for every point of $\uX_0(\F_{q^n})$ over which the morphism $\uY\to \uX_0$ has geometrically irreducible fiber, there exists an $\F_{q^n}$-point in this fiber.  In particular, every point in $\uX_0(\F_{q^n})\setminus \uC(\F_{q^n})$ lies in the image of $\uY(\F_{q^n})\to \uX_0(\F_{q^n})$.  By the easy part of the Lang-Weil bound,
$$|\uC(\F_{q^n})| = O(q^{n\dim\uC}) \le O(q^{n(\dim \uX_0-1)}).$$
Thus, the complement of the image of $\uY(\F_{q^n})\to \uX_0(\F_{q^n})$ has cardinality $o(q^{n\dim\uX})$, which implies (\ref{close}).

If  $\phi$ is not geometrically irreducible, then $[K_0:K] > 1$.  Let $K_1$ denote the Galois closure of $K_0/K$ in a fixed separable closure $\bar K$.
We choose $m$ so that $\F_{q^m}$ contains the algebraic closure of $\F_q$ in $K_1$.  If we are content to limit consideration to $\F_{q^n}$-points of $\uX$ and $\uX_0$, where $m$ divides $n$,
we may replace $\uX$ and $\uX_0$ by the varieties $\uX_{\F_{q^m}}$ and $(\uX_0)_{\F_{q^m}}$ respectively, obtained by base change.
This has the effect of replacing $K$, $K_0$, and $K_1$ by $K\F_{q^m}$, $K_0\F_{q^m}$, and $K_1 \F_{q^m}=K_1$ respectively.  Replacing $q$ by $q^m$,  
we may now assume that $\F_q$ is algebraically closed in $K_1$.


Now, $\Gal(K_1/K)$ acts faithfully on $A_1$ as $\F_q$-algebra.  As $A$ is integrally closed in $K$ and $A_1$ is the integral closure of $A$ in $K_1$, it follows that 
$$A\subset A_1^{\Gal(K_1/K)}\subset A_1\cap K = A,$$
so $A = A_1^{\Gal(K_1/K)}$; likewise, $A_0 = A_1^{\Gal(K_1/K_0)}$.
Geometrically, this means that $\uX$ and $\uX_0$ are the quotients of $\uX_1$ by $\Gal(K_1/K)$ and $\Gal(K_1/K_0)$ respectively.
We denote these quotient maps $\pi$ and $\pi_0$ respectively.  Thus we have the diagram
$$\xymatrix{\uX_1\ar[d]_{\pi_0}\ar@/^1pc/[dd]^{\pi}\\ \uX_0\ar[d]_{\psi}\\ \uX}$$

As the action of $\Gal(K_1/K)$ on $\uX_1$ is faithful and $\uX_1$ is irreducible, there is a dense affine open subvariety of $\uX_1$ on which $\Gal(K_1/K)$ acts freely.  
Replacing $\uX_1$ by this subvariety and $\uX$ and $\uX_0$ by quotients of this subvariety by $\Gal(K_1/K)$ and $\Gal(K_1/K_0)$ respectively affects 
$o(q^{n\dim \uX})$ of the $\F_{q^n}$-points of $\uX$, $\uX_0$, and $\uX_1$, so without loss of generality, we may assume that $\Gal(K_1/K)$ acts freely on $\uX_1$.
Now
\begin{equation}
\label{part}
\psi(\uX_0(\F_{q^n})) =  \psi(\uX_0(\F_{q^n})\setminus \pi_0(\uX_1(\F_{q^n}))) \cup \pi(\uX_1(\F_{q^n})).
\end{equation}
By Lang-Weil, $|\uX_1(\F_{q^n})| = (1+o(1))q^{n\dim\uX}$, so
\begin{align*}|\pi_0(\uX_1(\F_{q^n})))| &= ([K_1:K_0]^{-1}+o(1))q^{mn},\\
|\pi(\uX_1(\F_{q^n}))| &= ([K_1:K]^{-1}+o(1))q^{mn}.
\end{align*}
By (\ref{part}),
$$|\psi(\uX_0(\F_{q^n}))| \le (1-[K_1:K_0]^{-1}+[K_1:K]^{-1}+o(1))q^{mn},$$
which implies (\ref{not-onto}).
\end{proof}

\begin{lem}
\label{Jordan}
Let $G$ be a finite group acting transitively on a set $S$ with more than one element and $H$ a normal subgroup of $G$ such that every element of $H$ has at least one fixed point in $S$.
Then for all $s\in S$, $H\,\Stab_G(s)$ is a proper subgroup of $G$.
\end{lem}

\begin{proof}
By a classical theorem of Jordan, every non-trivial transitive permutation group contains a derangement, so $H$ must act intransitively.  Thus, the orbit of $H\, \Stab_G(s)$ containing $s$
is a proper subset of $S$, which implies the lemma.
\end{proof}

\begin{lem}
\label{Galois}
Let $K$ be a field, $\bar K$ a separable closure of $K$, and $K_1$ and $K_2$  finite extensions of $K$ in $\bar K$.  Suppose $K_1$ is Galois over $K$ and $K_2\neq K$.  
If $K_1\cap K_2 = K$, then there exists an element of $\Gal(\bar K/K_1)$ which does not stabilize any $K$-embedding of $K_2$ in $\bar K$. 
\end{lem}

\begin{proof}
Let $K_3$ be the Galois closure of $K_2$ in $\bar K$ and define $G := \Gal(K_1K_3/K)$.  Thus $G$ acts transitively on the set $S$ of $K$-embeddings of $K_2$ in $\bar K$.
Let $H = \Gal(K_1 K_3/K_1)$, which is normal in $G$ since $K_1/K$ is Galois.  If every element of $\Gal(\bar K/K_1)$ fixes at least one element of $S$, then by Lemma~\ref{Jordan},
$H\,\Stab_G(s)$ is a proper subgroup of $G$, where $s$ denotes the identity embedding of $K_2$ in $\bar K$.  If $L$ is the fixed field of $K_1 K_3$ under $H\,\Stab_G(s)$, then 
$L$ is a non-trivial extension of $K$
contained in both $(K_1 K_3)^H = K_1$ and $(K_1 K_3)^{\Stab_G(s)} = K_2$.

\end{proof}

\begin{prop}
\label{Cheb}
Let $\uX$ be a variety over $\F_q$ with coordinate ring $A$ with function field $K$.  Let $K\subset K_0,K_2\subset \bar K$, and let $K_1$ (resp. $K_3$) denote the Galois closure of $K_0$ (resp. $K_2$)
in $\bar K$.  Let $A_i$ for $0\le i\le 3$ denote the integral closure of $A$ in $K_i$, and let $\uX_i := \Spec A_i$.  If $K_1$ and $K_2$ satisfy the hypotheses
of Lemma~\ref{Galois}, then there exists $\epsilon > 0$ so that for all sufficiently large integers $n$, there are
at least $\epsilon q^{n\dim \uX}$ elements of $\uX(\F_{q^n})$ which lie in the image of $\uX_i(\F_{q^n})\to \uX(\F_{q^n})$ for $i=0$ but not for $i=2$.  
\end{prop}

\begin{proof}
Let $K_{13} = K_1 K_3$, $A_{13}$ denote the integral closure of $A$ in $K_{13}$, and $\uX_{13}$ denote $\Spec A_{13}$.  Let $G := \Gal(K_{13}/K)$.
The action of $G$ on $\uX_{13}$ is faithful, and $\uX_{13}$ is irreducible, so there exists a dense open affine subvariety $\uU_{13}\subset \uX_{13}$
on which $G$ acts freely.  Replacing $\uX_{13}$, together with its quotients by subgroups of $G$, by $\uU_{13}$ and its corresponding quotients affects only
$o(q^{n\dim\uX})$ $\F_{q^n}$-points of these quotients, and therefore does not affect the statement of the proposition.  We may therefore
assume that we are in the setting of \cite[Theorem~6]{Serre} and can apply the Chebotarev density theorem for varieties.

By Lemma~\ref{Galois}, there exists $g\in G$ such that $g$ acts trivially on $K_1$ but acts without fixed points on the set of $K$-embeddings $K_2\to \bar K$
or, equivalently, on the geometric points lying over any given geometric point of $\uX$ for the covering map $\uX_2\to \uX$.  This implies that if $x\in \uX(\F_{q^n})$ and $g$ belongs to the Frobenius conjugacy class of $x$, then
there is no $q^n$-Frobenius stable point lying over $x$ on $\uX_2\to \uX$, i.e., $x$ does not lie in the image of $\uX_2(\F_{q^n})\to \uX(\F_{q^n})$.  On the other hand,
every geometric point of $\uX_0$ lying over $x$ is stable by the $q^n$-Frobenius, so $x$ lies in the image of $\uX_2(\F_{q^n})\to \uX(\F_{q^n})$.  By  Chebotarev density 
\cite[Theorem~7]{Serre}, the proposition follows for every $\epsilon < |G|^{-1}$. 
\end{proof}

The main technical result of this section is the following.
\begin{prop}
\label{key}
Let $\phi\colon \uY\to \uX$ be a dominant morphism of normal varieties over $\F_q$.  Then there exist subsets $X_{n,i}\subset \uX(\F_{q^n})$, $1\le i\le m$, 
with the following property.
If $\theta \colon \uZ\to \uX$ is any dominant morphism of normal varieties over $\F_q$,
$\theta(\uZ(\F_{q^n}))\supset \phi(\uY(\F_{q^n}))$ for all $n$, and for some $n$, $\theta(\uZ(\F_{q^n}))\cap X_{n,i}$ is non-empty for each $i=1,\ldots,m$, then 
the generic fiber of $\theta$ is geometrically irreducible.
\end{prop}

\begin{proof}
Let $A$, $B$, $C$ denote the coordinate rings of $\uX$, $\uY$, and $\uZ$ respectively.  Let $K$, $L$, and $M$ be the fields of fractions of $A$, $B$, and $C$ respectively.
We regard $B$ and $C$ as $A$-algebras via $\phi$ and $\theta$ respectively, so $L$ and $M$ are extensions of $K$.
Let $K_0$ and $K_2$ denote the separable closures of $K$ in $L$ and $M$ respectively.  As $B$ and $C$ are finitely generated $\F_q$-algebras, $L$ and $M$ are finitely generated
$K$-extensions, and $K_0$ and $K_2$ are finite separable extensions of $K$.   The claimed generic irreducibility of the generic fiber of $\theta$ amounts to the equality $K=K_2$.
We define $\bar K$, $K_1$, $K_3$, and $K_{13}$ as in Proposition~\ref{Cheb}. 

Let $F_1,\ldots,F_m$ denote all subfields of $K_1$ over $K$, excluding $K$ itself.  Thus, we have the following diagram of fields:
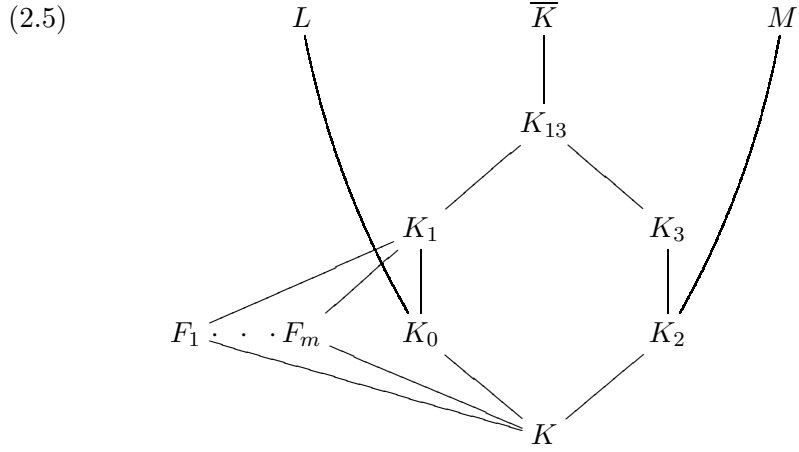
\begin{equation}
\label{fields}
\xymatrix{&L\ar@/_/@{-}[dddr]&&\bar K\ar@{-}[d]&&M\ar@/^/@{-}[dddl]\\ &&&K_{13}\ar@{-}[dl]\ar@{-}[dr] \\ &&K_1\ar@{-}[d]\ar@{-}[dl]\ar@{-}[dll]&&K_3\ar@{-}[d]\\ F_1\hbox{\rlap{$\ \cdot\ \,\cdot\ \,\cdot$}}\ar@{-}[drrr]&F_m\ar@{-}[drr]&K_0\ar@{-}[dr]&&K_2\ar@{-}[dl]\\ &&&K}
\end{equation}

For $0\le i\le 3$, let $A_i$ denote the integral closure of $A$ in $K_i$ and $\uX_i = \Spec A_i$; likewise for $A_{13}$ and $\uX_{13}$.
For $1\le i\le m$, let $D_i$ denote the integral closure of $A$ in the field $F_i$,
and let $\uW_i := \Spec D_i$.  
By (\ref{fields}), we have the following diagram of schemes:
$$
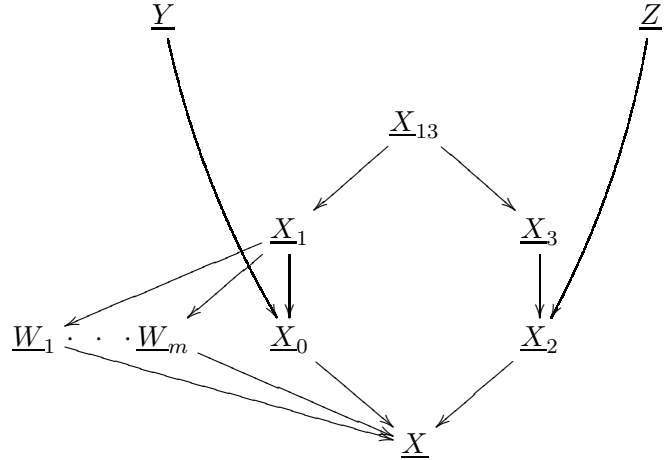
\xymatrix{&\uY\ar@/_/[dddr]&&&&\uZ\ar@/^/[dddl]\\ &&&\uX_{13}\ar[dl]\ar[dr] \\ &&\uX_1\ar[d]\ar[dl]\ar[dll]&&\uX_3\ar[d]\\ \uW_1\hbox{\rlap{$\ \cdot\ \,\cdot\ \,\cdot$}}\ar[drrr]&\uW_m\ar[drr]&\uX_0\ar[dr]&&\uX_2\ar[dl]\\ &&&\uX}$$

Let $X_{n,i}$ denote the complement of the image of $\uW_i(\F_{q^n})$ in $\uX(\F_{q^n})$.  We note for future use that by (\ref{not-onto}) and the Lang-Weil estimate, for $1\le i\le m$, 
\begin{equation}
\label{large-X}
|X_{n,i}| \ge \epsilon q^{\dim \uX} > \frac\epsilon 2|\uX(\F_{q^n})|
\end{equation}
if $n$ is sufficiently large.

Moreover, the condition $\theta(\uZ(\F_{q^n}))\supset \phi(\uY(\F_{q^n}))$ implies that 
\begin{align*}
|\im(\uX_1&(\F_{q^n})\to \uX(\F_{q^n})) \setminus \im(\uX_2(\F_{q^n})\to \uX(\F_{q^n}))| \\
&\le |\im(\uX_0(\F_{q^n})\to \uX(\F_{q^n})) \setminus \im(\uX_2(\F_{q^n})\to \uX(\F_{q^n}))| \\
&= |\im(\uY(\F_{q^n})\to \uX(\F_{q^n})) \setminus \im(\uZ(\F_{q^n})\to \uX(\F_{q^n}))| + o(q^{n\dim\uX}) \\
&= |\phi(\uY(\F_{q^n})) \setminus \theta(\uZ(\F_{q^n}))| + o(q^{n\dim\uX}) \\
&= o(q^{n\dim\uX}).
\end{align*}
If $K_2\neq K$, Proposition~\ref{Cheb} implies that $K_1\cap K_2$ must be a non-trivial extension of $K$, so $F_i\subset K_2$ for some $i\in [1,m]$.
Thus 
$$\theta(\uZ(\F_{q^n}))\subset \im(\uW_i(\F_{q^n})\to \uX(\F_{q^n})),$$
contrary to the assumption that $\theta(\uZ(\F_{q^n}))\cap X_{n,i}$ is non-empty for each $i$.
We conclude that $K_2=K$, and the proposition follows.

\end{proof}

\section{Random walks}

\begin{lem}
\label{Markov}
Let $G$ be a finite group and $S$ a (not necessarily symmetric) set of generators.  Let $\sS_1,\sS_2,\ldots$ be i.i.d. random variables on $G$ with support $S$.
Let $\sG_n = \sS_1\cdots\sS_n$.  Suppose that there does not exist a homomorphism from $G$ to any non-trivial cyclic group $C$ mapping $S$ to a single element.
Then the limit as $n\to \infty$ of the distribution of $\sG_n$ is the uniform distribution on $G$.
\end{lem}

\begin{proof}
Consider the Markov chain with state space $G$ in which the probability of a transition from $g$ to $hg$ is $\Pr[\sS_i = h]$.
Since the uniform distribution is stationary, it suffices to check that this Markov chain is irreducible and periodic \cite[Theorem~4.9]{LPW}.  Irreducibility is immediate from the condition 
that $S$ generates $G$.  If the Markov chain is periodic, then for some proper subset $X\subset G$ and some integer $m$, $s_1\cdots s_m \in \Stab_G(X)$ for all $s_i\in S$.
Let $G_m$ denote the subgroup of $G$ generated by 
$$\{s_1 \ldots s_m\mid s_1,\ldots, s_m\in S\}.$$
As $G_m \subset \Stab_G(X)\subsetneq G$, $G_m$ is a proper subgroup of $G$.

Consider the subgroup $\tilde G$ of $G\times \Z/m\Z$ generated by $\{(s,1)\mid s\in S\}$.  By definition, the kernel of projection on the second factor is $G_m$.  
By Goursat's Lemma, $\tilde G$ is the pullback to $G\times \Z/m\Z$ of the graph of an isomorphism between $G/G_m$ and a quotient of $\Z/mZ$.  This identifies
$G/G_m$ with a non-trivial cyclic group $C$, and all elements of $S$ map to the same generator of $C$, contrary to hypothesis.
\end{proof}

Let $\sX_n$ denote the sum of $n$ independent i.i.d. random variables on $\Z^2$, each uniformly distributed on $\{(\pm 1,0),(0,\pm1)\}$.

\begin{lem}
Let $p>2$ be prime, $k$ a positive integer, and $\epsilon > 0$.  For $n$ sufficiently large,
$$\Pr[\sX_n\in p^k \Z^2] < \frac{1+\epsilon}{p^{2k}}.$$
\end{lem}

\begin{proof}
The image under (mod $p^k$) reduction of our random walk on $\Z^2$ is a random walk on $G =(\Z/p^k\Z)^2$  
with generating set $S = \{\pm 1,0),(0,\pm 1)\}$.  As differences between elements of $S$ generate $G$,
there is no proper coset of $G$ which contains $S$.  By Lemma~\ref{Markov},  $\sX_n$ becomes uniformly distributed (mod $p^k$)
in the limit $n\to \infty$, which implies the lemma.

\end{proof}

\begin{lem}
Let $k$ be a positive integer, and $\epsilon > 0$.  For $n$ sufficiently large,
$$\Pr[\sX_n\in 2^k \Z^2] < \frac{2+\epsilon}{4^k}.$$
\end{lem}

\begin{proof}
If $n$ is odd, the probability that $\sX_n\in 2\Z^2$ is zero.  We therefore assume $n=2n'$, so $\sX_n$ is the sum of $n'$ i.i.d.
random variables supported on 
$$\{(\pm 2,0),(0,\pm2), (\pm1,\pm1), (0,0)\}.$$  
Reducing (mod $2^k$), we obtain an irreducible aperiodic 
random walk on $\ker (\Z/2^k\Z)^2\to \Z/2\Z$, and the argument proceeds as before by Lemma~\ref{Markov}.
\end{proof}

\begin{proof}
The random walk on $(\Z/p^k\Z)^2$ in which each step is uniformly distributed on $\{(\pm 1,0),(0,\pm1)\}$
defines an irreducible and aperiodic Markov chain, since the set of possible steps does not lie in a single coset of any proper subgroup of $(\Z/p^k\Z)^2$.
Therefore, the distribution after $n$ steps converges to the unique stationary distribution.  Since the uniform distribution is stationary, this must be the limit,
and the lemma follows.
\end{proof}

\begin{prop}
For all $\epsilon > 0$, there exist $M$ and $N$ such that for $n\ge N$,
$$\Pr[\sX_n\in \bigcup_{i > M} i\Z^2] < \epsilon.$$
\end{prop}

\begin{proof}
By \cite[Proposition~3.2]{LST}, if $p>2$ is prime,
$$\Pr[\sX_n\in p\Z^2\setminus\{(0,0)\}] < \frac 4{(p+1)^2}.$$
We choose $s\ge 2$ large enough that
$$\sum_{p > s} \frac 4{(p+1)^2} < \frac \epsilon2$$
and choose $k$ such that $\frac {3s}{4^k} < \frac \epsilon2$, so that if $n$ is sufficiently large, the total probability that $\sX_n\in p^k \Z^2$ for some $p\le s$ is less than $\epsilon/2$.
Note that this includes the probability that $\sX_n=(0,0)$
Let $M$ be larger than $s\prod_{p\le s} p^k$.  If $i>M$, then either $i$ has a prime factor greater than $s$ or a prime factor $\le s$ with multiplicity at least $k$.
The probability that there exists $i>M$ such that $G\in i\Z^2$ is therefore less than $\epsilon$.
\end{proof}

For any positive integer $d$ and non-negative integer $n$, we define $\sX_{d,n}$ to be the convolution of $n$ i.i.d. random variables on $\Z^d$, each uniformly distributed over the $2d$-element set consisting of
the standard generators and their inverses.

\begin{prop}
\label{gcd}
For all $d\ge 2$ and $\epsilon > 0$, there exist $M$ and $N$ such that for $n\ge N$,
$$\Pr[\sX_{d,n}\in \bigcup_{i > M} i\Z^d] < \epsilon.$$
\end{prop}

\begin{proof}
The projection of a random walk on $\Z^d$ onto the first two coordinates gives a random walk on $\Z^2$ where each of the four possible
non-zero steps are equally likely, but a zero step is also possible in the projection if $d>2$.  Since the projection of an element of $i\Z^d$ is an element of $i\Z^2$,
the conditional probability that $\sX_{d,n}\in \bigcup_{i > M} i\Z^d$ if we condition on at least $n'$ steps which are non-zero in the projection is less than $\epsilon/2$
if $n'$ is large enough.  Given $n'$ the probability that there are at least $n'$ steps non-zero in the projection goes to $0$ as $n$ goes to infinity, so it can be taken to be less than $\epsilon/2$, implying that $\Pr[\sX_{d,n}\in \bigcup_{i > M} i\Z^d] < \epsilon$.
\end{proof}

\begin{lem}
\label{G-walk}
Let $G$ be a perfect finite group and $N$ and $d$ positive integers.  Suppose $S = \{(g_{i1},\ldots,g_{iN})\mid 1\le i\le d\}$ is a $d$-element generating set of $G^N$.  Let $\sG_n$ denote the product of
$n$ i.i.d. random variables on $G$ which are supported and uniformly distributed on $S$.  Then the limiting distribution on $G^N$ as $n\to \infty$ is the uniform distribution.
\end{lem}

\begin{proof}
As $G$ is perfect, the same is true for $G^N$, and there is no non-trivial homomorphism to a cyclic group.  The lemma follows, therefore, from Lemma~\ref{Markov}.  
\end{proof}

\section{Proof of Theorem~\ref{main}}

%
%

We now prove the main theorem.

\begin{proof}
We fix a simple, simply-connected algebraic group $\uG$ over a finite field $\F_q$.
We will apply Proposition~\ref{key} in the case $\uX = \uG$, $\uY = \uG$, $\uZ=\uG^d$, $\phi$ is the $m$th power map for some positive integer $m$, and $\theta$ is the evaluation map $w$ for some $w\in F_d$ for which $\bar w = (a_1,\ldots,a_d)$ and $\gamma(a_1,\ldots,a_d) = m$.  
Writing $m=a_1b_1+\cdots+a_db_d$, we have
$$w_{\uG(\F_{q^n})}(g^{b_1},\ldots,g^{b_d}) = g^m$$
for all $g\in \uG(\F_{q^n})$, so $\phi(\uG(\F_{q^n})) \subset \theta(\uG(\F_{q^n}))$.

By the main theorem of \cite{MT}, for every finite simple group $\Gamma$, there exists a $2$-element generating set of $\Gamma^N$ whenever $N \le 2\sqrt{|\Gamma|}$.
Setting $N := q^n$ and applying 
this to $\Gamma := \uG(\F_{q^n})/Z(\uG(\F_{q^n}))$, we see that $\Gamma^N$ is $d$-generated.  As $G := \uG(\F_{q^n})^N$ is a perfect central extension of 
$\Gamma^N$, lifting any set of $d$ generators of the latter to the former, we again obtain a generating set.  

We denote by 
$$S = \{(g_{i1},\ldots,g_{iN})\mid 1\le i\le d\}$$
a generating set of $G$
and consider an $n$-step random walk on this group with generating set $S$.
By Lemma~\ref{G-walk}, for all $\delta > 0$, if  $n$ sufficiently large, the probability that the walk ends in any subset $T\subset G$ is at least 
$$(1-\delta/2)|T|/|G|.$$
We define $T := T_0\cup\cdots\cup T_{\lfloor N/m\rfloor-1}$, where 
$$T_i := \uG(\F_{q^n})^{im}\times X_{n,1}\times\cdots\times X_{n,m} \times \uG(\F_{q^n})^{N-(i+1)m},$$
and $X_{n,k}$ is defined as in Proposition~\ref{key}.

To estimate the probability that a uniformly randomly chosen element of $G$ lies in $T$, we note that
membership in the $T_i$ are independent conditions.  The probability of membership in each $T_i$ is
$$\prod_{j=1}^m \frac{|X_{n,j}|}{|\uG(\F_{q^n})|} \ge  \frac{\epsilon^m}{2^m}$$
by (\ref{large-X}).  Therefore, the probability of membership in $T$ for a uniformly chosen element of $G$ is at least 
$$1-(1-\epsilon^m/2^m)^{\lfloor N/m\rfloor}.$$
Taking $n$ (and therefore $N$) sufficiently large, we can guarantee this exceeds $1-\delta/2$.  Thus, the probability that the random walk ends in $T$ is greater than $1-\delta$.

For $1\le j \le N$, let $\bg_j = (g_{1j},\ldots,g_{dj})$.  We have seen that for a random word $w$ of length $n$, the probability that $(w(\bg_1),\ldots,w(\bg_N))\in T$ is greater than $1-\delta$.  
Membership in $T$ implies membership in some $T_i$, which implies
$$w(\bg_{im+1}) \in X_{n,1},\ldots,w(\bg_{im+m}) \in X_{n,m},$$
and therefore, by Proposition~\ref{key}, if $\gamma(\bar w) = m$,  $w$ is geometrically almost uniform for $\uG$.

Thus, for each $m$, there is at most a probability of $\delta$ that a random word $w$ of length $n$ satisfies $\gamma(\bar w) = m$ and that $w$ is not geometrically almost uniform.
By Lemma~\ref{gcd}, for each fixed $\epsilon > 0$, there exists $M$ such that if $n$ is large enough, the probability that $\gamma(\bar w)$ is zero or greater than $M$ is less than $\epsilon$.  Therefore, the probability that $w$ is not geometrically almost uniform for $\uG$ is less than $\epsilon + M\delta$.
Choosing first $\epsilon$ and then $\delta$, we can make this quantity as small as we wish, proving the theorem.

\end{proof}

We remark that the proof also shows that almost all words $w$ are almost uniform for the family of groups $\{\uG(\F_{q^n})\mid n\ge 1\}$.
The proof, together with that of  \cite[Theorem~2]{LST}, implies that $w$ is almost always uniform for all finite simple groups with fixed root system and characteristic.
For instance, almost all $w$ are almost uniform for the Suzuki and Ree groups.

\section{Questions}

\begin{question}
If $\cG$ is a simple, simply connected group scheme over $\Z$,
does the probability that a random word is almost uniform for all simple groups of the form $\cG(\F_q)/Z(\cG(\F_q))$ go to $1$?
\end{question}

It seems likely that the methods of this paper will allow one to prove this for all characteristics satisfying some Chebotarev-type condition, but can one do it for 
all characteristics simultaneously, or even a density one set of characteristics?  Even more optimistically, one can ask:

\begin{question}
\label{unbounded-rank}
Does the probability that a random word is geometrically almost uniform for all simple, simply connected algebraic groups over finite fields go to $1$?
\end{question}

Given an $e$-tuple of words $w_1,\ldots,w_e\in F_d$, for each $G$ we can define a function $G^d\to G^e$, and we can ask about almost uniformity.
In geometric families, this reduces again to the question of the geometric irreducibility of the generic fiber of the morphism $\uG^d\to \uG^e$
for simple, simply connected algebraic groups over finite fields.  In the case that 
$$\Z^d/\Span_{\Z}(\bar w_1,\ldots,\bar w_e) \cong \Z^{d-e},$$
the function
$\uG(\F_{q^n})^d\to \uG(\F_{q^n})^e$ is surjective.  Geometric irreducibility for such words follows as before.

\begin{question}
For $e<d$, does the probability that a random $e$-tuple of elements of $F_d$ of length $n$ is geometrically almost uniform go to $1$ as $n\to \infty$?
\end{question}

Question~\ref{unbounded-rank} has an analogue for simple, simply connected compact Lie groups.  As a special case, one can ask:

\begin{question}
Does the probability that for a random word $w$ of length $n$
$$\lim_{m\to \infty} \Vert w_*U_{\SU(m)^d} - U_{\SU(m)}\Vert = 0$$
go to $1$ as $n\to\infty$?
\end{question}

\end{document}